\documentclass[compress,final,3p,times,11pt]{elsarticle}
\usepackage{lineno,hyperref}
\modulolinenumbers[5]
\usepackage{caption}
\usepackage{sidecap}
\usepackage[english]{babel}
 \pdfoutput=1
\usepackage{amsfonts}
\usepackage{graphicx}
\usepackage{graphics}
\usepackage{amssymb}
\usepackage{amsmath}
\usepackage{color, epstopdf}
\usepackage{mathrsfs}

\newtheorem{proof}{Proof}

\newtheorem{theorem}{Theorem}
\newtheorem{remark}{Remark}

\usepackage{graphicx, color, epstopdf}
\usepackage{graphicx}
\usepackage{float}
\usepackage{subcaption}
\usepackage{threeparttable}
\usepackage{float}

\usepackage{array}
\usepackage{subcaption}
\journal{...}
\begin{document}
\begin{frontmatter}
\title{Mean Exit Time and Escape Probability for the Stochastic Logistic Growth Model with Multiplicative $\alpha-$Stable L\'evy Noise }

\author{Almaz Tesfay\footnote{Corresponding author}}

\address{School of Mathematics and Statistics \& Center for Mathematical Sciences,  Huazhong University of Science and Technology\\
Wuhan,430074,
China}

\address{Department of Mathematics, Mekelle University\\
 Mekelle,P.O.Box 231,
 Ethiopia\\
amutesfay@hust.edu.cn}
\author{Daniel Tesfay}
\address{School of Mathematics and Statistics \& Center for Mathematical Sciences,  Huazhong University of Science and Technology\\
Wuhan,430074,
China}

\address{Department of Mathematics, Mekelle University\\
 Mekelle,P.O.Box 231,
 Ethiopia\\
dannytesfay@hust.edu.cn}

\author{Anas Khalaf}
\address{School of Mathematics and Statistics \& Center for Mathematical Sciences,  Huazhong University of Science and Technology\\
Wuhan,430074,
China\\
anasdheyab@hust.edu.cn}

\author{James Brannan}
\address{Department of Mathematical Sciences, Clemson University\\
Clemson, South Carolina 29634,
 USA\\
jrbrn@clemson.edu}

\begin{abstract}
In this paper we formulate a stochastic logistic fish growth model driven by both white noise and non-Gaussian noise. We focus our study on the mean time to extinction, escape probability to measure the noise-induced extinction probability and the Fokker-Planck equation for fish population $X(t)$. In the Gaussian case, these quantities satisfy local partial differential equations while in the non-Gaussian case, they satisfy nonlocal partial differential equations. Following a discussion of existence, uniqueness and stability, we calculate numerical approximations of the solutions of those equations. For each noise model we then compare the behaviors of the mean time to extinction and the solution of the Fokker-Planck equation as growth rate $r$, carrying capacity $K$, intensity of Gaussian noise $\lambda$, noise intensity $\sigma$ and stability index $\alpha$ vary. The MET from the interval $(0,1)$ at the right boundary is finite if $\lambda < \sqrt 2$. For $\lambda > \sqrt 2$, the MET from $(0,1)$ at this boundary is infinite. A larger stability index $\alpha$ is less likely leading to the extinction of the fish population.
\end{abstract}

\end{frontmatter}
\section{Introduction}

A well-known model used to describe the growth or decline of a population $X(t)$ of a given biological species is the Verhulst \cite{Veysel F.(2018)} or logistic equation.
 \begin{equation}\label{Detmin}
  \frac{dX(t)}{dt}=rX(t)[1-\frac{X(t)}{K}].\tag{1.1}
 \end{equation}

In this paper we assume that $X(t)$ is the number of a certain species of fish (e.g., cod, herring, or anchovy) at time $t$ in a given area of the ocean. In this equation, the constant $r$ is referred to as the intrinsic growth rate. This is the growth rate of the population in the absence of any limiting factors. The constant $K$ is referred to as the carrying capacity of the population in the environment. This is the maximum population that the species can sustain indefinitely, given the habitat in which the species resides.

 The solution of Eq. (\ref{Detmin}) subject to the initial condition
$$  X(0)=x_{0}$$
 is
 \begin{equation}\label{SolDerm}
  X(t)=\frac{x_{0}K}{x_{0}+(K-x_{0})e^{-rt}}.\tag{1.2}
 \end{equation}

Equation (\ref{Detmin}) has two equilibrium solution, $X_u = \phi_{1}(t) = 0$, and $X_s = \phi_{2}(t) = K$. From the solution (\ref{SolDerm}) it follows that

 $$\lim_{t\rightarrow\infty}x(t)=K$$
 for any $X_{0}>0$. Hence $X_u=\phi_{1}(t)=0$ is unstable and $X_s=\phi_{2}(t)=K$ is asymptotically stable.

Sometimes it is informative to write Eq. (\ref{Detmin}) in the form
 $$ \frac{dX}{dt}=-U'(x),$$
where
$$  U(x)=-\int rx(1-\frac{x}{K})dx=-\frac{1}{2}rx^2+\frac{1}{3}\frac{r}{K}x^3$$
is the potential function for the fish population. The potential function has a local maximum at $X_u = 0$, and local minimum at $X_s = K$.

In biology, the unstable state $X_u$ corresponds to the fish free state (or the state of fish extinction), and the stable state $X_s$ corresponds to a nonzero constant fish population.

There are numerous environmental factors that affect both growth rate and carrying capacity of the fish population: food supply, predators, competing species, temperature and quality of the water, geographical constraints, disease, and so on. Since it is difficult or impossible to account for these factors in a simple model, it is useful to conceive of the fish population as a dynamical subsystem contained within a large enveloping system that we simply refer to as the environment. The influence of the environment on the subsystem is then accounted for in a population sense by treating coefficients and/or input to Eq. (\ref{Detmin}) as random variables or processes whose statistical properties are supposed to be known. The solution of the equation will be a random process, and the problem consists of finding its statistical properties as well as the statistical properties of certain functionals of the solution.

Under the effect of stochastic growth rate, the density of fish population \cite{D.M.JamesDifferentialDynamicSytem(2007)} will fluctuate on the state of stable $D=(0,\infty).$ In section (\ref{4}), we study evolution of the fish growth density in the domain $D$ and the extinction probability induced by stochastic fluctuations.

Recent works on the stochastic logistic growth  model  are mostly concerned with  the model under Gaussian noise \cite{{Jonathan H(2014)},{Md.Asaduzzaman Shah2013},{A.Tsoularis2001},{Peter Kink (2018)},{X.Mao2002}}\label{page 2} and the references therein. The logistic growth systems with L\'evy noise have attracted some  recent attention  \cite{{M.Liu(2015)},{Applebaum D(2009)},{X.Zhang(2014)},{Ruihua W.(2014)}}. In fact, L\'evy noise appears to be  more realistic than Gaussian noise, due to jumps by excitatory and inhibitory impulses caused by external disturbances in biological systems.

L. Meng and Z. Baichuan \cite{M.Liu(2015)} observed that the stochastic logistic equation driven by Brownian motion the probability of extinction is zero (permanent), but when the authors considered the L\'evy noise, they observe that the population is extinctive if the intensity of the  L\'evy jump  is greater than a threshold, but the population still is permanent if the intensity of the  L\'evy jump  is less than the threshold in their study.

In our present paper  we consider the logistic model of a fish population under the perturbations of ( non-Gaussian) L\'evy noise as well as (Gaussian) Brownian  noise,
 \begin{equation}\label{SDEBMLM1}
  dX(t)=rX(t)(1-\frac{X(t)}{K})dt+\lambda X(t) {dB(t)}+\sigma X(t)dL_{t}^{\alpha}, \qquad X(0)=x_0, \tag{1.3}
  \end{equation}
where $r>0$, is the growth rate of the fish population, $K>0$ is the carrying capacity of the environment,$ B(t)$ is the standard Brownian motion, $\lambda$ represents the intensity of Gaussian noise, $L_{t}^{\alpha}$ is an $\alpha$-stable L\'evy motion and $\sigma $ is the noise intensity. Here the noise is  multiplicative because the diffusion and the intensity coefficient depend on X(t).

In this study, we will consider the escape problem for (\ref{SDEBMLM1}). More concretely, we will study whether the system trajectory starting from the stable equilibrium point in Eq. (\ref{SDEBMLM1}) reaches other region through a boundary under the influence of $\alpha$-stable L\'evy noise. To analyse the problem we consider three different deterministic quantities that carry dynamical information of the SDE in (\ref{SDEBMLM1}). These deterministic quantities include mean first exit time (MET); escape probability (EP) and probability density function (PDF) of the Fokker-Planck equation (FPE) for the solution paths. Fortunately, these deterministic quantities can be determined by solving the nonlocal partial differential equation in the case of L\'evy noise and local partial differential equations in the case of White noise (Sec. 4). Then we numerically  calculate MET, EP and FPE of the solution stating from the escape region to the various outside region. We also examine how these quantities depend on the parameter $r$ ( the growth rate), stability index $\alpha$, and the noise intensity $\sigma $.

The organization of the paper is as follows. In Sec. \ref{2}, we introduce preliminary concepts briefly. We show that the solution of our model is exits and positive under certain conditions in Sec. \ref{3}. In Sec. \ref{4}, we define the three deterministic quantities i.e. MET, EP and FPE together with appropriate region for computing these quantities. Numerical results and biological implications of Gaussian white noise case $(\sigma = 0)$  and non-Gaussian noise case $(\lambda = 0)$ are given in Secs. \ref{5} and \ref{6} respectively. In Secs. \ref{7} and \ref{8}, we give the results and conclusions of our study.

\section{Preliminaries} \label{2}
In this section, we recall basic concepts and facts that we will need throughout our study.
{\emph{\subsection{\textbf{{Brownian motion:}}}}

Brownian motion is stochastic process ( adapted process) $\{\mathbf{B}(t),\mathfrak{F}_t; t\geq0\}$ defined on a complete probability space $(\Omega, \mathfrak{F},\mathfrak{F}_t,\mathbf{P})$. Brownian motion is named after the botanist R. Brown (1773-1858), who in 1827 studied the motion of tiny particles suspended in water. He observed that the particles moved in an erratic random fashion \cite{Grimmett}. The Brownian process, $B(t)$, satisfies the following conditions:
\cite{Karatzas}
\begin{itemize}
\item  $B(0)=0$ a.s,
\item $B(t)$ is independent increment and stationary,
\item $B(t)-B(s)$ is normally distributed with mean 0 and variance $t-s$, for $0<s<t$.i.e. $B(t)-B(s)\sim N(0,t-s)$,
\item The trajectories of $B(t)$ are nowhere differentiable and it has a continuous sample paths with probability one.
\end{itemize}
\emph{\subsection{\textbf{$\alpha$-Stable L\'evy process:}}}
A stable distribution $S_{\alpha}(\mu,\beta,\gamma)$ is the distribution for a stable random variable \cite{Applebaum Book(2009)}, where the stability index $\alpha \in (0,2)$, the skewness $\mu \in (0,\infty)$ and the shift $\gamma \in (-\infty,\infty)$. An $\alpha$-stable Levy process $L_{t}^{\alpha}$ is a non-Gaussian stochastic process satisfying the following conditions.
\begin{itemize}
\item  $L_{0}^{\alpha}=0$, a.s;

\item  $L_{t}^{\alpha}$ has stationary increments: $L_{t}^{\alpha}$-$L_{s}^{\alpha}$ and $L_{t-s}^{\alpha}$ have the same distribution $S_{\alpha}((t-s)^1/{\alpha},0,0)$;
\item $L_{t}^{\alpha}$ has independent increments: $0\leq t_{1}<t_{2}<t_{3}<...<t_{i-1}<t_{i}<\infty,$ the random variables $L_{t_{i+1}}^{\alpha}-L_{t_{i}}^{\alpha}$ are independent for each $i=1,2,...$;
\item $L_{t}^{\alpha}$ has stochastic continuous sample path: Sample paths are continuous in probability. In other words for all $\delta>0$, all $s\geq0$; the skewness $P(|L_{t}^{\alpha}-L_{s}^{\alpha}|>\delta) \rightarrow 0$ as $t\rightarrow s$.
\end{itemize}

A L\'evy process $L_t^{\alpha}$ taking values in $R=(-\infty, \infty)$ is characterized by a drift vector $\hat{b}\in R $, a non negative constant $\hat{Q}$ and a Borel measure $\nu$ defined on ${R}\setminus{\{0\}}$. The triplet $(\hat{b},\hat{Q},\nu)$ is called the generating triplet of L\'evy motion $L_t^{\alpha}$. According to the L\'evy -It$\hat{o}$ decomposition $L_t^{\alpha}$ \cite{duan2015introduction} can be expressed as
\begin{equation}\label{Ito-Dec}
 L_t^{\alpha}=\hat{b}t+ B_{\hat{Q}}(t)+\int_{|y|<1}y\tilde{N}(t,dy)+\int_{|y|\geq 1}yN(t,dy), \tag{2.1}
\end{equation}
where $N(t,dy)$ is the independent  Poisson random measure on $R^+\times{{R}\setminus{\{0\}}}$, $\tilde{N}(t,dy)=N(t,dy)-\nu(dy)dt$ is the compensated Poisson random measure, $\nu(S)=E(N(1,S))$ is the jump measure, and $B(t)$ is the an independent standard 1-dimensional Brownian motion.

The  L\'evy-Khinchin formula for any  L\'evy process has a specific form for its characteristic function. In other words for $0\leq t<\infty$, $u \in R$,
\begin{equation*}
\mathbb{E}[e^{(iu{L_t})}] = e^{(t\psi(u))},
\end{equation*}
where
$$\psi(u)=iu\hat{b}-\frac{\hat{Q}}{2}u^2 + \int_{{R}\setminus{\{0\}}}(e^{iuz}-1-iuzI{_{|z|<1}})\nu(dz).$$
 where $\nu$ is a L\'evy measure which is defined by
 $$\nu(du)=c(1,{\alpha})\frac{1}{|u|^{1+\alpha}}du $$}
 where $ c(1,\alpha) = \alpha\frac{\Gamma(\frac{1+\alpha}{2})}{{2^{1-\alpha}\pi^{\frac{1}{2}}}\Gamma{(1-\frac{\alpha}{2})}}$ and $\Gamma$ is the Gamma function. The function of the L\'evy measure is to describe the expected number of jumps of a certain size at a time interval 1. Usually, the parameter $\alpha$ is called the index of stability with the value $0 < \alpha < 2$.

In the case of a one-dimensional $\alpha $-stable L\'evy motion, the drift vector $\hat{b}=0$, the diffusion $\hat{Q}=0$.

In our paper, we focus on jump process with a specific size in generating triplet $(0,0,\nu_{\alpha})$ for the random variable $S_{\alpha}$ which can be defined by $\Delta L_{t}^{\alpha}=L_{t}^{\alpha}-L_{t^-}^{\alpha}<\infty, t \geq 0,$ where $L_{t-}^{\alpha}$ is the left limit of the Le\'vy motion in $ R=(-\infty, \infty)$ at any time $t$.
\section{Existence and uniqueness of the positive solution}\label{3}

In investigating the dynamical behavior of a logistic growth model, we will show the existence and uniqueness of the positive solution the stochastic differential equation. Noting that $X(t)$ of the SDE in (\ref{SDEBMLM1}) denotes the size of fish population, so it should be positive. To guarantee that the SDE has a unique solution for a given initial value $x{_0}$, the coefficients of the equation are generally required to satisfy both the local lipschitz condition and linear growth condition.

 In this work, we will focus on the logistic fish growth model given in equation (\ref{SDEBMLM1}). According to Eq. (\ref{Ito-Dec}), we can rewrite Eq. (\ref{SDEBMLM1}) as follows
\begin{align}\label{SDEBMLM2}
  &dX(t)=f(X(t))dt+g(X(t)){dB(t)}+\int_{|y|<1}h(X(t))y\tilde{N}(dt,dy)\nonumber\\
  &+ \int_{|y|\geq 1}h(X(t))yN(dt,dy), \qquad X(0)=x_0.\tag{3.1}
  \end{align}
By D. Applebaum's book  \cite{Applebaum Book(2009)}, the large jump in the term (\ref{SDEBMLM2}) is omitted and our study focus with small jumps, so we can modify Eq. (\ref{SDEBMLM2}) as:
\begin{equation}\label{SDEBMLM3}
  dX(t)=f(X(t))dt+g(X(t)){dB(t)}+\int_{|y|<1}h(X(t))y\tilde{N}(dt,dy) \qquad X(0)=x_0.\tag{3.2}
  \end{equation}

The integral form of the SDE in (\ref{SDEBMLM3}) is
 \begin{align*}
 &X(t)=X(0)+\int_{0}^{t}f(X(s))ds+\int_{0}^{t}g(X(s)){dB(s)}\nonumber\\
 &+\int_{0}^{t}\int_{|y|<1}h(X(s))y\tilde{N}(ds,dy),
  \end{align*}
where
  $ f(X(t))=rX(t)(1-\frac{X(t)}{K})$ is a deterministic vector field, $g(X(t))=\lambda X(t)$ is diffusion coefficient with intensity of Gaussian noise $\lambda$, and $h(X(t))=\sigma X(t)$ is the noise intensity term, and $\sigma $ is the noise intensity.

Before we state the exitance and uniqueness theorem, we need to look the following assumptions on the vector field $f(x)$, diffusion coefficient $g(x)$ and noise intensity term $h(x)$.\\
\textbf{Assumption 1.} (Local Lipschitz condition): \cite{{Applebaum Book(2009)},{duan2015introduction},{Fima C2005introduction}}
 The terms $f, g $ and $\sigma$ satisfy the locally Lipschitz condition if $\forall T$, $\forall N > 0$, $\forall |x_j|,\leq N$,  for $ j=1,2,\forall t\in[0,T],$ there exists a positive $\tau$ such that\\
 $|f(x_1)-f(x_2)|^2 + |g(x_1)-g(x_2)|^2 +\int_{|y|<1}|h(x_1)y-h(x_2)y|^2\nu(dy)\leq \tau|x_1-x_2|^2,$  for all $y \in R.$\\
\textbf{Assumption 2.}
(Linear Growth Condition): \cite{{Applebaum Book(2009)},{duan2015introduction},{Fima C2005introduction}} $\forall T, \forall N > 0, \forall |x| \leq N$ there exists $L > 0$ , $ \forall t\in[0,T],$ such that\\
$|f(x)|^2 +|g(x)|^2 + \int_{|y|<1}|h(x)y|^2\nu(dy)\leq L (1+|x|^2),$   for all $ y\in R.$\\
\textbf{Assumption 3.}
The function $f(x)$ is continuous in $x\in R_+=(0,\infty)$.\\
\textbf{Assumption 4.}
A function $h(x,y)=h(x)y$ is a measurable function and $x \mapsto h(x,y)$ is continuous for $y\in\{y:1\leq|y|\}$.

\begin{theorem}\label{Thrm1}
( Existence and Uniqueness of the solution $X(t)$) (J.Duan,\cite{duan2015introduction}):
 If Assumption 1 - Assumption 4 hold, then the SDE in (\ref{SDEBMLM3}) with the standard initial condition has a unique  global solution $X(t)$. The solution $X(t)$ is adapted and c\'ad\'ag.  Assume ${\mathrm{E}} \|x_0\|^2 < \infty,$ then ${\mathrm{E}} \|X(t)\|^2 < 0,$ and there exits a positive $k(t)$ for all $t>0$ such that
 $${\mathrm{E}} \|X(t)\|^2 < k(t)(1+\|x_0\|^2).$$
\end{theorem}
\begin{proof}
Since the coefficients of the SDE are locally Lipschitz continuous for any $x_0$, there is a unique local solution $X(t)$ on $[0,T]$, where $T$ is the explosion time. We need to show this solution is global, i.e, (to show $T=\infty)$ . The proof this theorem is similar to D. Applebaum (\cite{Applebaum Book(2009)}, Theorem 6.2.3), J. Duan (\cite{duan2015introduction}, Theorem 7.26), and X. Zhang \cite{X.Zhang(2014)}.
\end{proof}
\begin{remark}
The stochastic differential equation SDE given in Eq.(\ref{SDEBMLM3}) satisfies the above Assumptions  and Theorem (\ref{Thrm1}). Thus SDE has a unique positive solution.
\end{remark}
\begin{remark}
  The generator of the stochastic differential equation \cite{{M.L.Hao(2014)},{Yanjie Zhang(2020)}} in (\ref{SDEBMLM3}) is given by
\end{remark}
\begin{align}\label{Generator}
  &A\varphi(x)=f(x)\varphi'(x)+\frac{1}{2}{g(x)}^2\varphi''(x)\nonumber\\
  &+|h(x)|^{\alpha}\int_{\mathbb{R}\setminus{\{0\}}}[\varphi(x+z)-\varphi(x)-z\varphi'(x)I_{|z|<1}(z)]\nu(dz). \tag {3.3}
   \end{align}
\section{Deterministic quantities }\label{4}
In this section, we present numerical schemes for solving three deterministic quantities: mean exit time (MET), escape probability (EP) and the Fokker-Plank equation (FPE).
\emph{\subsection{\textbf{Mean exit time (MET)}}}
Consider the initial value problem\\
\begin{equation*}
dX_t=f(X_t)dt+g(X_t)dB(t)+\sigma(X_t)dL_t^\alpha, \qquad  X_0=x_0 \in (0,K)
\end{equation*}
where $K > 0$ is the carrying capacity of the fish population.\\
The mean exit time (MET) $u(x)\geq 0,$ for an orbit starting at $x$, from the domain $D$ denoted by
\begin{equation*}
u(x)=\mathbb{E}(inf\{{ t{\geq}0 :X_t({\omega}, x)\} \in  D^c)} , \qquad X_0=x
\end{equation*}
is helpful to quantify the dynamic behaviors of the SDE driven by the symmetric $\alpha$-stable L\'evy process. Here $D^c$ is the compliment of the set $D$ in $R$.\\
 The MET $u(x)$ satisfies the following integral-differential equation. \cite{Brannan J.1999}\\
 \begin{equation}\label{MET01}
 Au(x)=-1,\qquad x\in D,\tag{4.1}\\
 \end{equation}
\begin{equation}\label{MET02}
   u(x)=0,\qquad x\in D^c,\tag{4.2}
\end{equation}
where the generator $A$ is
 \begin{align}\label{MET03}
 &Au(x)=f(x)u'(x)+\frac{1}{2}{g(x)}^2u''(x)\nonumber\\
 &+|h(x)|^{\alpha}\int_{\mathbb{R}\setminus{\{0\}}}(u(x+z)-u(x)-zu'(x)I_{|z|<1}(z))\nu(dz)=-1,\tag{4.3}
\end{align}
for $x\in D.$ Equation (\ref{MET02}) is a non-local Dirochelet condition for the exterior interval $D^c$.

The solution to Eq. (\ref{MET01})-(\ref{MET02}) gives the mean exit time for the fish population to either become extinct if exit occurs at $x=0$ or recovery to its carrying capacity if exit occurs at $x=K$.

\begin{remark}
The numerical simulation of Eq. (\ref{MET01}), Eq. (\ref{MET02}) and Eq. (\ref{MET03}) can be done by a similar method to that in \cite{gao2014mean}.
\end{remark}
\emph{\subsection{\textbf{Escape probability (EP)}}}
In this subsection we present how to quantify EP of the dynamic progression of the stochastic differential equation in (\ref{SDEBMLM2}). Let us start by defining it.

The likelihood fish population $X_t$, starting at a point x$_0$ in the domain D=(0,K), exits $D$ in a finite time and lands in a subset $ E\subseteq D^c$ is called escape probability (EP). The EP denoted by $P_{E}(x)$, satisfies the differential-integral equation \\

\begin{equation*}
Ap_{E}(x) = 0,  \qquad x\in D,
\end{equation*}
with Dirichlet boundary condition\\
\[
P_{E}(x)=
\begin{cases}
1, \qquad x\in E,\\
0, \qquad x\in D^c/ E,
\end{cases}
\]
where $A$ is defined in Eq. (\ref{Generator}).\\

In our work, we are interested in the effect of noise on extinction probability, so we take $E=(-\infty, 0]$. Because in this interval the fish population goes to extinct.

\subsection{\emph{\textbf{Fokker-Plank equation (FPE)}}}
The Fokker-Plank equation(FPE) is an important deterministic tool for quantifying the behavior of a stochastic dynamic system.
The FPE of the SDE driven by non-Gaussian noise only in Eq. (\ref{SDEBMLM3}) in terms of the probability density function $P(x,t)$, for the solution $X_t$ with the given initial condition $X(0)=x_0$, \cite{{T. Daniel(2019)},{T.Gao2016}}satisfies\\
\begin{equation*}
 p_t(x,t)=A^{*}p(x,t), \qquad x\in D,
 \end{equation*}
\begin{equation*}
p(x,0)=\delta(x-x_0), \qquad x_0\in D,
\end{equation*}
where $\delta $ is the  dirac function and $A^*$, the adjoint operator of A in Hilbert space $L^2(R)$, obtained by solving
$$\int_{\mathbb{R}\setminus{\{0\}}}A\varphi (x)V(x)dx=\int_{\mathbb{R}\setminus{\{0\}}}\varphi (x)A^*V(x)dx,$$
for $\varphi $ ,V in the domain of definition for the operator A and $A^*$. We find that $$A^*V(x)=\int_{\mathbb{R}\setminus{\{0\}}}[|h(x+z)|^{\alpha}V(x+z)-|h(x)|^{\alpha}V(x)]\nu(dz).$$
Therefore we have  \cite{Yanjie Zhang(2020)}
\begin{align}\label{FPEG}
&\frac{\partial p}{\partial t}=-\frac{\partial}{\partial x}(f(x)p(x,t))+\frac{1}{2}\frac{\partial ^2}{\partial x^2}(g(x)^2p(x,t))\nonumber\\
&+\int_{\mathbb{R}\setminus{\{0\}}}[|h(x+z)|^{\alpha}p(x+z)-|h(x)|^{\alpha}p(x)]\nu (dz).\tag{4.4}
\end{align}
\section{ Gaussian white noise case ($\sigma=0)$}\label{5}

Consider the standard stochastic logistic growth equation driven by Gaussian noise \cite{K Lundquist(2011)}.
\begin{align}\label{SDEMB1}
d\tilde{X}(\tilde{t})=r\tilde{X}(\tilde{t})(1-\frac{\tilde{X}(\tilde{t})}{K})d\tilde t+\tilde\lambda \tilde{X}(\tilde{t})dB(\tilde{t}) \qquad  X(0)=x_0.\tag{5.1}
 \end{align}
Now let's non-dimensionlize Eq. (\ref{SDEMB1}) using the following scalings. Define\\

\qquad \qquad $t=\tilde{t} r$, \qquad $X=\frac{\tilde{X}}{K},$  and $\lambda=\frac{\tilde{\lambda}}{\sqrt r}$.\\
Since $\mathbb{E}[dB(\frac{t}{r})]^2=\frac{dt}{r}=\mathbb{E}[\frac{1}{\sqrt r}dB(t)]^2.$ In other words $dB(\frac{t}{r})= \frac{1}{\sqrt r}dB(t)$.

 Using these scaling, Eq. (\ref{SDEMB1}) is transformed into

\begin{equation}\label{SDEMB2}
dX(t)=X(t)[1-X(t)]dt+\lambda X(t)dB(t),\qquad  X_0=x.\tag{5.2}
 \end{equation}

\emph{\subsection{\textbf{Exact solution of the stochastic differential equation}}}

Applying Ito's formula to $Y(t) = X^{-1}(t)$ \cite {Md.Asaduzzaman Shah2013} gives the linear initial value problem for $Y(t)$

\begin{equation}\label{ItoForm}
dY(t)=[(\lambda^2 - 1)Y(t) + 1]dt - \lambda Y(t)dB(t),\tag{5.3}
 \end{equation}
 $$ Y(0)=\frac{1}{x_0}.$$
According to Duan's book (\cite{duan2015introduction}, Example 4.22) the Eq. (\ref{ItoForm}) is linear stochastic differential equation with $ a_1=\lambda^2 - 1, a_2=1, b_1= - \lambda ,b_2=0$. ( In fact in (\cite{duan2015introduction} $a_1, a_2, b_1$ and $b_2$ are time depended).\\
The solution of Eq. (\ref{ItoForm}) is \cite{{P.E.Kloeden},{V.M(2011)}}

$$Y(t)= \phi(t) \{Y _0 +\int_{0}^{t} e^{(-\frac{1}{2}\lambda^2+1)s+\lambda B(s)}ds \}.$$\\
where $\phi(t)=e^{(\frac{1}{2}\lambda^2-1)t-\lambda B(t)}$ is the fundamental solution.\\
Consequently, the unique, strong solution of Eq. (\ref{SDEMB2})is
\begin{equation}\label{FinalSol}
 X(t)=Y^{-1}(t)=\frac {x_0e^{(1-\frac{\lambda^2}{2})t+\lambda B(t)}}{{1+x_0 \int_{0}^{t} e^{(1-\frac{\lambda^2}{2})s+\lambda B(s)}}ds}.\tag{5.4}
\end{equation}
From Eq. (\ref{FinalSol}), we see that the solution exists for all $t>0$ and if $x_0 >0$, then $X(t) > 0$ a.s. for all $t>0$. Let's rewrite Eq.(\ref{FinalSol}) in the form

\begin{equation}\label{FinalSol2}
X(t)= \frac{x_0e^{(1-\frac{\lambda^2}{2})t \left(1+\frac{\lambda B(t)}{{(1-\frac{\lambda^2}{2}) t}}\right)}}{1+x_0 \int_{0}^{t} e^{(1-\frac{\lambda^2}{2})s+\lambda B(s)}ds}            \tag{5.5}
 \end{equation}
According to V. Mackevičius's book (\cite{V.M(2011)}, Sec.11.4), we have
\begin{equation}\label{StrongNum}
\lim_{t\rightarrow\infty}\frac{B(t)}{t}=0,  \qquad a.s.\tag{5.6}
\end{equation}
If $x_0>0$ and $\lambda>\sqrt 2$, it follows from Eq.(\ref{FinalSol2}) and (\ref{StrongNum}) that\\
\begin{equation*}
\lim_{t\rightarrow\infty}X(t)=0, a.s.
\end{equation*}
Note that
$$\lim_{\lambda\rightarrow 0}X(t)=\frac{x_0}{x_0+e^{-t}(1-x_0)},$$
the solution of the deterministic logistic equation.

\subsection{\textbf{The Fokker-Planck equation and its stationary density}}

The transition density function $p(t,y/x)$ for the process $\{X(t),t>0\}$ as in Eq. (\ref{FPEG}) satisfies the equation
\begin{equation}\label{fpe}
\frac{\partial p}{\partial t}=-\frac{\partial}{\partial x}[x(1-x)p]+\frac{1}{2}\lambda^2\frac{\partial^2}{\partial x^2}[x^2p].\tag{5.7}
\end{equation}

The stationary density $q(x)=\lim _{t \rightarrow \infty} p(x,t/x_0)$, if it exists, satisfies the second order ODE
\begin{equation}\label{stratBM}
-\frac{\partial}{\partial x}[x(1-x)p]+\frac{1}{2}\lambda^2\frac{\partial^2}{\partial x^2}[x^2p]=0.\tag{5.8}
\end{equation}

Equation (\ref{stratBM}) has two linearly independent solutions\\

                  \qquad  $q_1(x)=x^{2(1-1/\lambda^2)}e^{-2x/\lambda^2}$  \qquad and \qquad $q_2(x)=1,$
so the general solution of Eq. (\ref{stratBM}) is
$$q=c_1q_1(x)+c_2q_2(x).$$

The requirement that $\int^{\infty}_{0} q(x)dx=1$ implies that $c_2=0$ and
\begin{equation}\label{fpe2}
q=\frac{x^{2(1/\lambda^2-1)}e^{ -2x/\lambda^2}}{\int_{0}^{\infty}x^{2(1/\lambda^2-1)}e^{ -2x/\lambda^2}dx}.   \tag{5.9}                                                  \end{equation}
provided the integral $\int_{0}^{\infty}x^{2(1/\lambda^2-1)}e^{ -2x/\lambda^2}dx $ exits. For value of $x$ near 0 the approximation\\
\begin{figure}
  \begin{center}
    \includegraphics[width=0.6\linewidth]{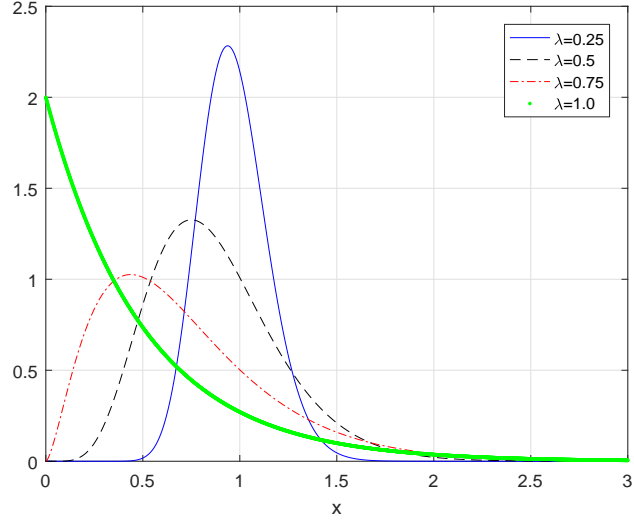}
    \caption{The stationary density of FPE of Eq. (\ref{fpe}) for $\lambda = 0.25, 0.5, 0.75,$ and 1.0.}
    \label{FPE1}
    \end{center}
  \end{figure}

\begin{equation*}
\int_{0}^{x}\eta^{2(1/\lambda^2-1)}d\eta = x^{2(\frac{1}{\lambda^2}-1)}
\end{equation*}
shows that the integral is finite if and only if $\lambda < \sqrt 2$. We note that
\begin{equation*}
\lim_{\lambda\rightarrow\sqrt 2}q(x;\lambda)=\delta(x).
\end{equation*}
The parameter $\lambda= \tilde{\lambda}/\sqrt{r}$ is the ratio of dispersion intensity to the square root of the
growth rate.  When $\sqrt{r}$ is large relative to $\tilde{\lambda}$ the stationary density is unimodal and peaks near the carrying capacity $x=\tilde{x}/K=1$, but as noise intensity $\tilde{\lambda}$ increases relative to $\sqrt{r}$ the stationary probability mass density skews left toward $x=0$.  When $\lambda$ is sufficiently large the term $X(t)(1-X(t))$ in Eq.~(5.2) that controls drift towards $x=1$ becomes increasingly negligible relative to dispersive noise $\lambda X(t)dB(t)$.  However, the presence of $X(t)$ in $\lambda X(t)dB(t)$
in effect causes the region near $x=0$ to be a region of stagnation for sample trajectories.  Although
$X(t)$, subject to large fluctuations may in the course of time achieve large values of $x$, the majority of time is spent in the region near the origin.

Whenever $0<\lambda<\sqrt{2}$ and $X(0)=x_0>0$, the fact that a stationary density exists suggests the following asymptotic behavior of the solution $p(x,t|x_0; \lambda)$ of the Fokker-Planck equation
\[
\lim_{t \rightarrow \infty}p(x,t|x_0; \lambda) = q(x;\lambda),
\]
and that the probability of exit from the domain $D=(0,\infty)$ at $x=0$ in finite time is 0.
However, in the limiting case $\lambda = \sqrt{2}$, the fact that $q(x;\sqrt{2}) = \delta(x)$ suggests that
all solutions ultimately exit at $x=0$ in finite time.
From Section 5.1 we also know that when $\lambda > \sqrt{2}$,
$\lim_{t \rightarrow \infty} X(t) = 0$ a.s.  In population biology terms, we conclude that for the model under discussion,
when $\lambda \ge \sqrt{2}$, all populations ultimately become extinct,
but when $0< \lambda < \sqrt{2}$, extinction cannot occur. ( see Figure \ref{FPE1}).

\subsection{\textbf{Probability of exit and mean exit time}}
For all values of $\lambda>0$ a fundamental set of solutions of
\begin{equation}\label{singulareqn}
\frac{1}{2} \lambda^2 x^2 u'' + x(1-x) u' = 0 \tag{5.10}
\end{equation}
is
\[
u_1(x) = 1 \quad \mbox{and} \quad u_2(x;L) = -\int_x^L \eta^{-2/\lambda^2} e^{2 \eta/ \lambda^2}\;d\eta.
\]
It is of interest to note the following infinite series representations for nonconstant solutions of
Eq. (\ref{singulareqn})
\begin{eqnarray*}
u_3(x) & = &x^{1-2/\lambda^2}e^{2x/\lambda^2}\left\{1 -\frac{2}{\lambda^2}\frac{1}{2-2/\lambda^2} x
+ \left(\frac{2}{\lambda^2}\right)^2\frac{1}{2-2/\lambda^2}\cdot\frac{1}{3-2/\lambda^2}x^2+\cdots \right\}
\nonumber\\ & & \mbox{if } 2/\lambda^2 \notin \left\{1,2,3, \ldots \right\},
\end{eqnarray*}
\begin{align*}
&u_4(x) = \sum_{n=0}^{M-2} \frac{M^n}{(n-M+1)n!}x^{n-M+1} + \frac{M^{M-1}}{(!M-1)} \ln x \nonumber\\
&+ \sum_{n=M}^{\infty} \frac{M^n}{(n-M+1)n!}x^{n-M+1},
\qquad \mbox{if } 2/\lambda^2 = M \in \left\{2,3,4, \ldots \right\},
\end{align*}
and
\begin{equation*}
u_5(x) = \ln x + \sum_{n=1}^{\infty} \frac{1}{n\cdot n!}x^n \quad \mbox{if } 2/\lambda^2 =1.
\end{equation*}
Note that
\begin{equation*}
u_2(x;L) = \left\{ \begin{array}{cl}
u_3(x)-u_3(L), & 2/\lambda^2 \notin \left\{1,2,3, \ldots \right\}, \\
u_4(x) - u_4(L), & 2/\lambda^2 = M \in \left\{2,3,4, \ldots \right\},\\
u_5(x) - u_5(L), & 2/\lambda^2 =1.
\end{array}\right.
\end{equation*}
It follows that $u_2(x;L)$ is singular when $2/\lambda^2 \ge 1$, that is, when $\lambda \le \sqrt{2}$, with
asymptotic behavior
\[
u_2(x;L) \sim \left\{ \begin{array}{cl} x^{1-2/\lambda^2} \mbox{ as } x \rightarrow 0, & \lambda < \sqrt{2}, \\
\ln x \mbox{ as } x \rightarrow 0, & \lambda = \sqrt{2}.
\end{array}\right.
\]
as $x \rightarrow 0$.
When $\lambda > \sqrt{2}$, a fundamental set of solutions of Eq. (\ref{singulareqn}) is
$\left\{u_1(x), u_3(x) \right\}$ where $u_3(x)$ is not singular at $x=0$.  Its asymptotic behavior at the
origin is given by
\[
u_3(x) \sim x^{1-2/\lambda^2}\rightarrow 0, \quad x \rightarrow 0.
\]
\emph{\subsubsection{Probability of exit at x=0}}\label{ExitProbSection}
Given that $0 < \epsilon< L$ and that $X(0)=x \in D_\epsilon = (\epsilon, L)$, the probability $P(x)=P(x;\epsilon,L)$
of exit at $x=\epsilon$ before exit at $x=L$ satisfies the boundary value problem
\begin{equation*}
\frac{1}{2} \lambda^2 x^2 P''+x(1-x)P' = 0, \qquad P(\epsilon)=1, \quad P(L)=0.
\end{equation*}
Substituting the general solution $P(x) = c_1u_1(x) +c_2u_2(x)$ into the boundary conditions gives
\begin{equation*}
P(x;\epsilon,L) = \frac{\int_x^L \eta^{-2/\lambda^2} e^{2 \eta/ \lambda^2}\;d\eta}{\int_\epsilon^L \eta^{-2/\lambda^2} e^{2 \eta/ \lambda^2}\;d\eta}.
\end{equation*}
 We want to examine the behavior of $P(x;\epsilon,L)$ as $\epsilon\rightarrow 0$. Convergence of the integral $\int_{\epsilon}^{L}\eta^{{-2}/{\lambda^2}}e^{ {2\eta}/{\lambda^2}}d\eta$ is examined by approximating the integral for small value of $x$.

  \[ \int_{\epsilon}^{x}\eta^{{-2}/{\lambda^2}}e^{ {2\eta}/{\lambda^2}}d\eta\approx\int_{\epsilon}^{x}\eta^{{-2}/{\lambda^2}}d\eta=
\begin{cases}
\frac{1}{1-\frac{2}{\lambda^2}}[x^{1-\frac{2}{\lambda^2}}-\epsilon^{1-\frac{2}{\lambda^2}}], \qquad \lambda\neq\sqrt 2,\\
\ln (x)-\ln (\epsilon),     \qquad      \qquad     \lambda=\sqrt 2,
\end{cases}
\]

It follows that if $\lambda\leq\sqrt 2$, then the integral in the denominator diverges so
\begin{equation*}
\lim_{\epsilon \rightarrow 0} P(x;\epsilon,L) = 0.
\end{equation*}
In other words, if $0< \lambda \le \sqrt{2}$, for each $L>0$, starting from $x\in(0,L)$ the probability of exit at $x=0$ in finite time is zero, that is, this boundary is not accessible.  Then, starting from
$x \in (0, L)$, the probability of hitting the
right boundary $x=L$ in finite time is 1, and it makes sense to compute the expected time to hit $x=L$.
We do this in the next subsection.

If, on the other hand $\lambda > \sqrt{2}$
\begin{equation*}
\lim_{\epsilon \rightarrow 0}\int_\epsilon^L \eta^{-2/\lambda^2} e^{2 \eta/ \lambda^2}\;d\eta
=\int_0^L \eta^{-2/\lambda^2} e^{2 \eta/ \lambda^2}\;d\eta < \infty
\end{equation*}
so
\begin{equation*}
\lim_{\epsilon \rightarrow 0}P(x;\epsilon,L) = P(x;0,L) = \frac{\int_x^L \eta^{-2/\lambda^2} e^{2 \eta/ \lambda^2}\;d\eta}{\int_0^L \eta^{-2/\lambda^2} e^{2 \eta/ \lambda^2}\;d\eta}
= 1 - \frac{\int_0^x \eta^{-2/\lambda^2} e^{2 \eta/ \lambda^2}\;d\eta}{\int_0^L \eta^{-2/\lambda^2} e^{2 \eta/ \lambda^2}\;d\eta}>0.
\end{equation*}
In this case, for each $L>0$, there is a positive probability of exit at 0 before exit at $L$.  Due to the fact
that
\[
\lim_{L \rightarrow \infty} \int_0^L \eta^{-2/\lambda^2} e^{2 \eta/ \lambda^2}\;d\eta = \infty
\]
it follows that
\[
\lim_{L\rightarrow \infty}P(x;0,L) = 1.
\]
Thus, if $\lambda > \sqrt{2}$, starting at any $x>0$, the probability of exit from $D=(0, \infty)$ at $x=0$ is equal to 1.  In this model, even though populations may become large, they all ultimately become extinct due to high intensity noise.
\subsubsection{Probability of exit at $x=1$}
Computing the mean exit time at $x=1$ only makes sense if the probability of exit at $x=1$ in finite
time is equal to 1.
We verify this in this subsection, again in the case that $\lambda<\sqrt{2}$.  Let
$P(x;\epsilon)$ equal the probability of exit at $1$ before exit at $\epsilon$ given that $X(0)=x\in D_{\epsilon} = (\epsilon,1)$.

Then $P(x;\epsilon)$ satisfies the boundary value problem
\begin{equation}\label{ProbExit}
\frac{1}{2}\lambda^2 x^2 P'' + x(1-x)P'= 0, \quad \epsilon<x <1, \qquad P(\epsilon;\epsilon) = 0, \quad P(1;\epsilon) = 1. \tag{5.11}
\end{equation}
Substituting the general solution $P(x;\epsilon) = d_1 u_1(x) + d_2u_2(x)$ into the boundary conditions
gives
\begin{equation}\label{SolProbExit1}
P(x;\epsilon) = -\frac{u_2(\epsilon)}{u_1(\epsilon)u_2(1)-u_1(1)u_2(\epsilon)}u_1(x)+\frac{u_1(\epsilon)}{u_1(\epsilon)u_2(1)-u_1(1)u_2(\epsilon)}u_2(x).\tag{5.12}
\end{equation}
Using the facts that $u_1(x)=1$ and $u_2(1)=0$ this simplifies to
\begin{equation}\label{SolProbExit2}
P(x;\epsilon) = 1-\frac{1}{u_2(\epsilon)}u_2(x).\tag{5.13}
\end{equation}
Since $u_2(\epsilon) \rightarrow -\infty$ as $a \rightarrow 0$ when $\lambda<\sqrt{2}$, we find that
\begin{equation}
P(x) = \lim_{a \rightarrow 0}P(x;\epsilon) = 1, \qquad 0<x<1.\tag{5.14}
\end{equation}
Thus if $X(0)=x \in (0,1)$, the probability of exit at the right boundary in finite time is equal to one and it
makes sense to compute the expected exit time.
\emph{\subsubsection{Expected exit times}}
Consider a fish population that has been reduced by over harvesting or disease from its carrying capacity $K$ to a level $x<K$.  We would then be interested in the expected fish population recovery time, that is, the average time it takes the fish population to increase to the level $K$, or to a small neighborhood of $K$.
In dimensionless variables we set up the problem on the domain $D=(\epsilon,1)$ where the carrying capacity is equal to $1$ (recall dimensionless population is $x=\tilde{x}/K$).   The boundary value problem
for the expected exit time $u(x)$ is
\begin{equation} \label{METEq}
\frac{1}{2} \lambda^2 x^2 u''(x) +x(1-x)u'(x) = -1, \quad x \in D_\epsilon = (\epsilon, 1), \tag{5.15}
\end{equation}
with boundary conditions
\begin{equation}\label{METBCs}
u(\epsilon) = 0, \quad u(1) = 0. \tag{5.16}
\end{equation}
We first consider the case $\lambda < \sqrt{2}$.  From Sec. \ref{4}, we know that in this case
the probability of exit from the domain $D=(0,1)$ is equal to zero, so we are guaranteed that in the
limiting case of $\epsilon \rightarrow 0$ that exit will occur at the right boundary, $x=1$.
The general solution of Eqs. (\ref{METEq})-(\ref{METBCs}) is
\begin{equation}\label{GenSol}
u(x;\epsilon) = c_1(\epsilon)u_1(x) + c_2(\epsilon)u_2(x) + Y(x), \tag{5.17}
\end{equation}
where $Y(x)$ is a particular solution of Eq. (\ref{METEq}).  A particular solution of Eq. (\ref{METEq})
can be found by assuming the infinite series representation
\begin{equation} \label{YSeries}
Y(x) = a_0 \ln x + a_1 x + a_2 x^2 + \cdots.\tag{5.18}
\end{equation}
The form of this series is found by using the variation of parameters representation for a particular
solution of Eq. (\ref{METEq}).  Substituting the series (\ref{YSeries}) into Eq. (\ref{METEq}) and matching
coefficients of like powers of $x$ gives
\begin{eqnarray*}
a_0 & = & \frac{1}{1-\lambda^2/2}\\
a_1 & = & a_0 \\
a_2 & = & \frac{1}{2(1+\lambda^2\cdot 1/2)}\;a_1 \\
a_3 & = & \frac{2}{3(1+\lambda^2\cdot 2/2)}\;a_2 \\
\vdots & = & \vdots \\
a_{n+1} & = & \frac{n}{(n+1)(1+\lambda^2\cdot n/2)}\;a_n \\
\vdots & = & \vdots \\
\end{eqnarray*}
Thus
\begin{equation}\label{Series}
Y(x) = - \frac{1}{1-\lambda^2/2} \left(\ln x + x \right) + \sum_{n=2}^\infty a_n x^n. \tag{5.19}
\end{equation}
Substituting the general solution (\ref{GenSol}) into the boundary conditions (\ref{METBCs}) gives
\begin{align}\label{VarParm1}
&u(x;\epsilon)=\frac{u_2(\epsilon)Y(1)-u_2(1)Y(\epsilon)}{u_1(\epsilon)u_2(1)-u_1(1)u_2(\epsilon)}u_1(x)+\frac{u_1(1)Y(\epsilon)-u_1(\epsilon)Y(1)}{u_1(\epsilon)u_2(1)-u_1(1)u_2(\epsilon)}u_2(x) \nonumber\\
& + Y(x). \tag{5.20}
\end{align}
Using $u_1(x) = 1$ and $u_2(1) = 0$ this reduces to
\begin{equation}\label{VarParm2}
u(x;\epsilon)= Y(x) - Y(1) + \left[Y(1)-Y(\epsilon)\right]\frac{u_2(x)}{u_2(\epsilon)}.\tag{5.21}
\end{equation}
Now we let $\epsilon \rightarrow 0$ in Eq. (\ref{VarParm2}).  Since $u_2(\epsilon)\rightarrow \infty$
and $\ln \epsilon/u_2(\epsilon) \rightarrow 0$ we find that
\begin{equation}\label{YxminusY1}
u(x) = \lim_{\epsilon \rightarrow 0} u(x;\epsilon) = Y(x)-Y(1)\tag{5.22}
\end{equation}
is the solution of
\begin{equation}\label{METDomain01}
\frac{1}{2} \lambda^2 x^2 u''(x) + x(1-x)u'(x) = -1, \qquad x \in D=(0,1), \quad u(1)=0 \tag{5.23}
\end{equation}
in the case that $\lambda<\sqrt{2}$.  In Figure (\ref{MeanExitTimeCaption}), using parameter values $\lambda=1$ and $\epsilon=0.001$, we compare a numerical approximation to the solution of Eqs. (\ref{METEq})-(\ref{METBCs}) (dashed red curve) to the solution ({YxminusY1}) of problem (\ref{METDomain01})
(solid blue curve) in which we used a truncated series approximation of $Y(x)$ in Eq. (\ref{YSeries}).
Starting from $x \in D=(0, \infty)$, we note that the expected exit time goes to infinity as
$x \rightarrow 0$ when $0<\lambda<\sqrt{2}$.
\begin{figure}[h!]
\begin{center}
\includegraphics[width=10cm]{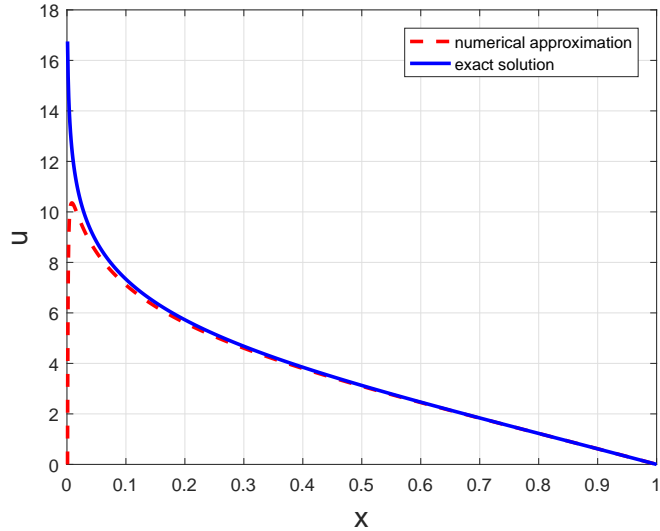}
\caption{Using $\lambda=1$, a comparison of the numerical approximation of the boundary value problem
Eqs. (\ref{METEq}) and (\ref{METBCs}) on $(\epsilon,1)$ with $\epsilon=0.001$ with a truncation series approximation of Eq. (\ref{METDomain01}) to the solution of problem (\ref{YSeries}).}
\label{MeanExitTimeCaption}
\end{center}
\end{figure}
In Section \ref{ExitProbSection} we showed that if $\lambda > \sqrt{2}$, starting at any $x>0$, the probability of exit from $D=(0, \infty)$ at the left boundary is equal to 1.  Populations may become large, but ultimately they all ultimately become extinct due to high intensity noise.  The solution of
\begin{equation} \label{METEq2}
\frac{1}{2} \lambda^2 x^2 u''(x) +x(1-x)u'(x) = -1, \quad x \in D_\epsilon = (\epsilon, L), \tag{5.24}
\end{equation}
with boundary conditions
\begin{equation}\label{METBC2s}
u(\epsilon) = 0, \quad u(L) = 0.\tag{5.25}
\end{equation}
is given by
\begin{align}\label{VarParm3}
&u(x;\epsilon,L)=\frac{u_3(\epsilon)Y(L)-u_3(L)Y(\epsilon)}{u_1(\epsilon)u_3(L)-u_1(L)u_3(\epsilon)}u_1(L)+\frac{u_1(L)Y(\epsilon)-u_1(\epsilon)Y(L)}{u_1(\epsilon)u_3(L)-u_1(1)u_3(\epsilon)}u_3(x)\nonumber\\
&+ Y(x). \tag{5.26}
\end{align}
It can be shown that
\[
\lim_{L\rightarrow \infty} u(x;\epsilon,L) = -Y(\epsilon)\left[1+Y(x)/Y(\epsilon)\right]
\]
and consequently
\[
\lim_{\epsilon \rightarrow 0} \left[ \lim_{L\rightarrow \infty}u(x;\epsilon,L)\right] = \infty.
\]
Thus, when $\lambda>\sqrt{2}$, the probability of exit from the interval $(0, \infty)$ at the left boundary is
1, but the expected exit time is infinity.

\section{Non-Gaussian L\'evy noise case $(\lambda=0$)}\label{6}
In this section, we consider the stochastic model driven by symmetric $\alpha$-stable L\'evy process. We focus on three issues: Mean exit time, escape probability, and Fokker-Plank equation to quantify the stochastic dynamics in (\ref{SDELM}).

Consider the following stochastic logistic model driven by non-Gaussian noise

\begin{equation}\label{SDELM}
 d\tilde{X} = r\tilde{X}(1-\frac{\tilde{X}}{K})dt+\sigma \tilde{X}dL_t^\alpha,\tag{6.1}
 \end{equation}
 $$ \tilde{X_0}=x_0$$

 \subsection{\textbf{Exact solution of SDE driven by non-Gaussian noise}}
  Now, let's non-dimensionlize Eq. (\ref{SDELM}) using the following scaling. Setting $X=\frac{\tilde{X}}{K}$.
 Using this scaling, Eq. (\ref{SDELM}) is transformed into
\begin{equation}\label{SDELM2}
dX(t)=rX(t)[1-X(t)]dt+\sigma X(t)dL_t^\alpha,\tag{6.2}
 \end{equation}
 $$X_0= x=\frac{ \tilde{X_0}}{K}.$$
Here, the vector field and noise intensity of SDE in (\ref{SDELM2}) satisfy Assumption 1 and Assumption 3 and this stochastic differential equation also satisfies theorem \ref{Ther} which means the SDE in (\ref{SDELM2}) has a unique positive solution.
\begin{theorem} \label{Ther}
Suppose that $r$ and $\sigma$ are positive real constant. Then there exists a unique solution $X(t)$ to Eq. (\ref{SDELM2}) for any initial value $X_0
> 0$, which is given by
\end{theorem}
\begin{equation}
X(t) = \frac{e^{-(r+\sigma \int_{|u|\leq 1}\frac{u^2+2u}{1+u}\nu(du) )t-\sigma \int_{|u|\leq 1}^t\frac{u}{1+u}\tilde{N}(ds,du)}}{\frac{1}{X_0}-\int_{|u|\leq 1}^t re^{-(r+\sigma \int_{|u|\leq 1}\frac{u^2+2u}{1+u}\nu(du) )t-\sigma \int_{|u|\leq 1}^t\frac{u}{1+u}\tilde{N}(ds,du)}ds}.\tag{6.3}
 \end{equation}
\begin{proof}
First let's rewrite Eq. (\ref{SDELM2}) in the form of Eq.(\ref{SDEBMLM3}) as follows
\begin{equation*}
  dX(t)=rX(t)[1-X(t)]dt+\sigma \int_{|u|<1}X(t))u\tilde{N}(dt,dy) \qquad X(0)=x_0.
  \end{equation*}
Set $Y(t)=-\frac{1}{X(t)}$, apply It\^o formula to $F(x)=-\frac{1}{x}$, where $Y(t)=F(x)$.
More details of the proof of this Theorem see [Z.Huang, J.Cao (2018) \cite{Zaitang H(2018)}, Theorem 3.1].
\end{proof}
\subsection{\textbf{Mean exit time} }

Mean exit time $u(x)$ is the expected time for the fish population $X(t)$ to either become extinct if exit occurs at $x=0$ or recovery to its carrying capacity if exit occurs at $x=K$.
According to equation Eq. (\ref{MET01}), here we present a numerical scheme to solve the following nonlocal partial differential equation, in order to get the mean exit time.
\begin{align}\label{METLM1}
 Au(x) = 0, \qquad x\in D=(0,K),\tag{6.4}
 \end{align}
$$u(x)=0, \qquad x\in D^c,$$
where $A$  is the generator in Eq. (\ref{Generator}).

Let's describe the numerical algorithms of equation (\ref{METLM1}) the scheme in the paper \cite{gao2014mean}. For simplicity, we use $D=(r_1,r_2)$ instead of $D=(0,K)$, so Eq. (\ref{METLM1}) becomes
\begin{align}\label{METLM3}
 rx(1-x)u'(x)+|h(x)|^{\alpha}\int_{\mathbb{R}\setminus{\{0\}}}(u(x+z)-u(x)-zu'(x)I_{|z|<1}(z))\nu(dz)=-1,\tag{6.5}
  \end{align}
for $x\in D=(r_1,r_2)$; $u(x)=0$ for $x\in D^c$.

In is paper, we choose $\delta=min\{|r_1-x|,|r_2-x|\}$. Thus, we obtain the following result:
\begin{align}\label{METLM4}
 &rx(1-x)u'(x)-|\sigma (x)|^{\alpha}\frac{C_{\alpha}}{\alpha}[\frac{1}{(x-r_1)^\alpha}+\frac{1}{(r_2-x)^\alpha}]u(x)\nonumber\\
 &+C_{\alpha}|\sigma (x)|^{\alpha}\int_{r_1-x}^{r_2-x}\frac{u(x+z)-u(x)}{|z|^{1+\alpha}}(dz)=-1,\tag{6.6}
  \end{align}
 where $\sigma (x)=\sigma x$, for $x \in(r_1,r_2)$, and $u(x)=0$ for $x\in D^c$.
 \begin{figure}[h!]
\begin{subfigure}[b]{0.5\linewidth}
  \includegraphics[width=\linewidth]{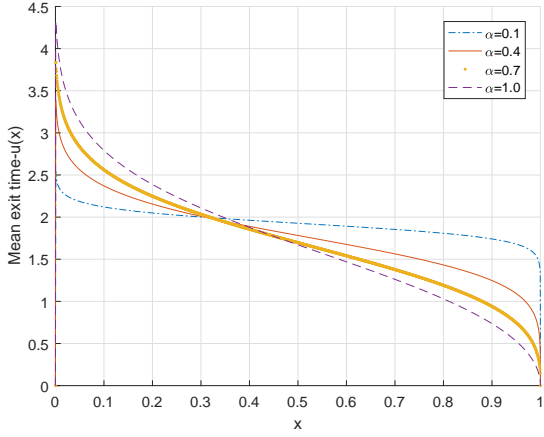}
 \caption{ The MET for $\alpha= 0.1, 0.4, 0.7, 1.0$, $\sigma=0.5$.}
     \end{subfigure}
     \begin{subfigure}[b]{0.5\linewidth}
  \includegraphics[width=\linewidth]{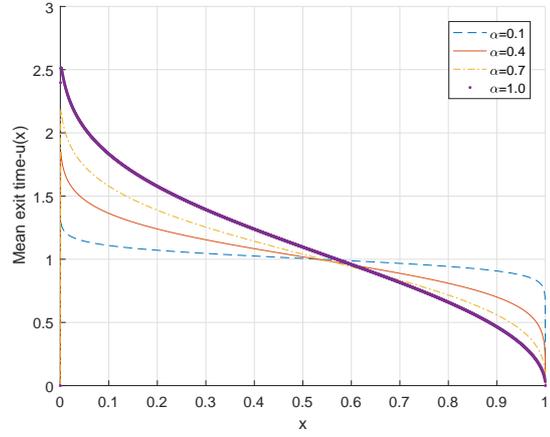}
 \caption{ The MET) for $\alpha= 0.1,0.4, 0.7, 1.0$, $\sigma=1$.}
     \end{subfigure}
     \begin{subfigure}[b]{0.5\linewidth}
  \includegraphics[width=\linewidth]{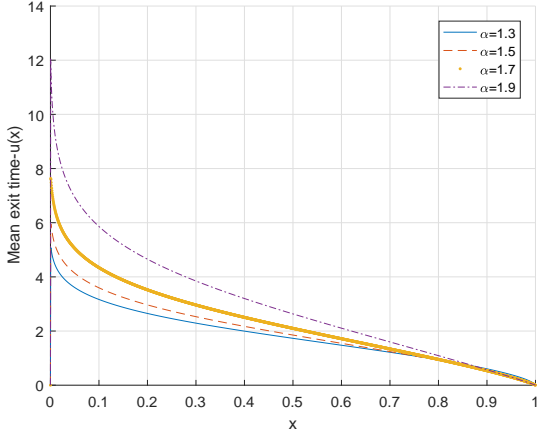}
 \caption{ The MET) for $\alpha= 1.3, 1.5, 1.7, 1.9$, $\sigma=0.5$.}
     \end{subfigure}
     \begin{subfigure}[b]{0.5\linewidth}
  \includegraphics[width=\linewidth]{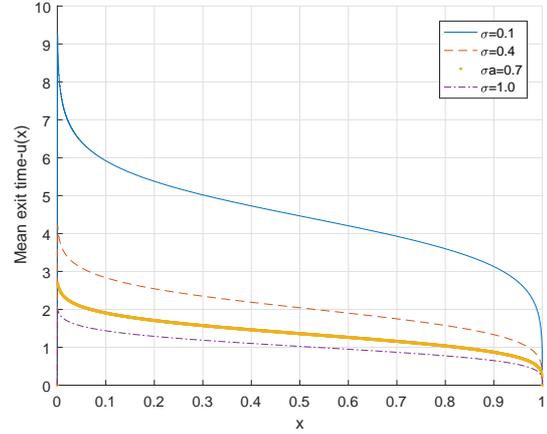}
 \caption{ The MET for $\sigma \in (0,1) $, and $\alpha =0.5$}
     \end{subfigure}
  \caption{The behavior of the mean exit time of Eq. (\ref{SDELM2}) for different values of the paraments, (a) $\sigma$ = 0.5, $\alpha \in (0,1)$ , (b) $r$ = 0.1, $\sigma$ = 1, $\alpha \in (0,1)$,(c) $\alpha \in (1,2)$,  $\sigma=0.5$ (d) $\sigma \in (0,1)$, $\alpha = 0.5$, $r=1.0$.}
  \label{MET2}
 \end{figure}
 Noting that $u$ is not smooth at the boundary point $x=r_1, r_2$, so in order to ensure the integral is smooth, so according the paper \cite{M.L.Hao(2014)}, we can rewrite Eq. (\ref{METLM4}) as:
  \begin{align}\label{METLM5}
 &rx(1-x)u'(x)-|h(x)|^{\alpha}\frac{C_{\alpha}}{\alpha}[\frac{1}{(x-r_1)^\alpha}+\frac{1}{(r_2-x)^\alpha}]u(x)\nonumber\\
 &+C_{\alpha}|h(x)|^{\alpha}\int_{r_1-x}^{r_1+x}\frac{u(x+z)-u(x)}{|z|^{1+\alpha}}(dz)\nonumber\\
 &+C_{\alpha}|h(x)|^{\alpha}\int_{r_1+x}^{r_2-x}\frac{u(x+z)-u(x)-z'u(x)}{|z|^{1+\alpha}}(dz)=-1,\tag{6.7}
  \end{align}
for $x\geq (r_1+r_2)/2$
 \begin{align}\label{METLM6}
 &rx(1-x)u'(x)-|h(x)|^{\alpha}\frac{C_{\alpha}}{\alpha}[\frac{1}{(x-r_1)^\alpha}+\frac{1}{(r_2-x)^\alpha}]u(x)\nonumber\\
 &+C_{\alpha}|h(x)|^{\alpha}\int_{r_1-x}^{r_1+x}\frac{u(x+z)-u(x)}{|z|^{1+\alpha}}(dz)\nonumber\\
 &+C_{\alpha}|h(x)|^{\alpha}\int_{r_1+x}^{r_2-x}\frac{u(x+z)-u(x)-z'u(x)}{|z|^{1+\alpha}}(dz)=-1,\tag{6.8}
  \end{align}
for $x<(r_1+r_2)/2$.
 The solution of Eq. (\ref{METLM5}) and Eq. (\ref{METLM6}), i.e, solution of the mean exit time can be obtained by applying the discretization method which is given in the paper \cite{gao2014mean}.

The numerical results of the MET in the non-Gaussian noise is given in figure (\ref{MET2}a) -(\ref{MET2}d). With a fixed value of the noise intensity $\sigma=0.5$, and the fish population $x \in(0,0.3)$, the result shows that the mean exit time is smaller with a larger value of the stability index ${\alpha}$. While density of the fish population $x \in (0.35, 1)$, the phenomenon is opposite, i.e, MET increases with the increase $\alpha$. The interval (0.3,0.35) is a transition period. In figure \ref{MET2}(b) shows that for the initial density of the fish population $x \in (0,0.55)$, the MET increases with increases in the stability index ${\alpha}$. While $x \in (0.6, 1)$, MET decreases with the increase $\alpha$ with the value of $\sigma=1.0$. The interval (0.55,0.6) is a transition period. In the case, $\alpha \in (1,2)$ and $\sigma=0.5$, the MET is larger with a larger value of ${\alpha}$, (Fig. \ref{MET2}(c)). The MET decreases with increases in the noise intensity $\sigma \in (0,1)$. This implies the fish population is sustainable for a large value of the stability index ${\alpha}$ with the growth rate $r$ and $\sigma$ are fixed. But in the initial density of the fish population $x \in (0,0.3)$ (Fig. \ref{MET2}(a), and in the case $x \in (0.6, 1)$ ( Fig. \ref{MET2}(b)), and with increases in the noise intensity $\sigma$ the fish population moves towards extinction.
\begin{figure}[h!]
    \centering
     \begin{subfigure}[b]{0.45\textwidth}
         \centering
         \includegraphics[width=\linewidth]{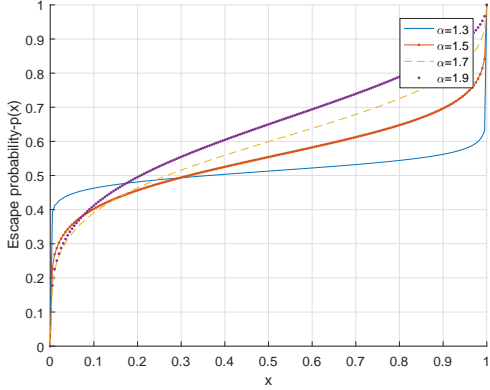}
    \caption{The EP for the stability index $\alpha \in (0,1).$ }
  \end{subfigure}
     \begin{subfigure}[b]{0.45\textwidth}
         \centering
         \includegraphics[width=\linewidth]{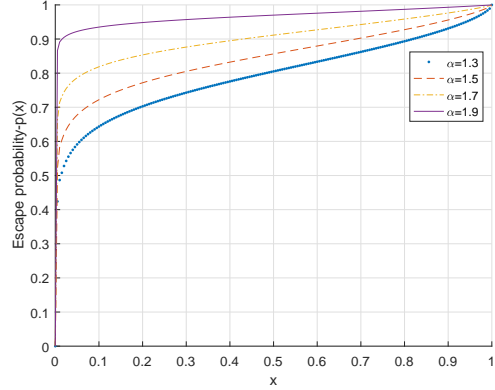}
     \caption{The  behaviour EP for $\alpha \in (1,2).$ }
  \end{subfigure}
     \begin{subfigure}[b]{0.5\textwidth}
         \centering
        \includegraphics[width=\linewidth]{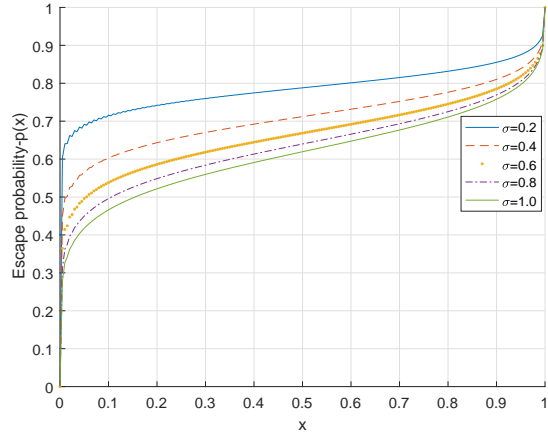}
    \caption{The EP when $\sigma= 0.2, 0.4, 0.6, 0.8, 1.0.$ }
  \end{subfigure}
   \caption{The escape probability of Eq. (\ref{SDELM2}) with fixed $r=1.0$ and $\sigma$ and $\alpha$ vary. (a) $\sigma$ = 1, $\alpha\in (0,1)$  (b) $\sigma$=1, $\alpha \in (1,2)$  (c) $\sigma\in (0,1)$ ,$\alpha = 0.5$.}
    \label{EP2}
\end{figure}
\subsection{\textbf{Escape probability}}

The escape probability of the stochastic differential equation in (\ref{SDELM2}) satisfies the following nonlocal partial differential equation:
\begin{align}\label{EPLM2}
 Ap_E(x) = 0, \qquad x\in D{_1}= (-1,1)\tag{6.9}
 \end{align}

  \[
p_E(x)=
\begin{cases}
1, \qquad x\in E\\
0, \qquad x\in D_1^{c}/ E
\end{cases}
\]
where $A$ is the generator of defined in equation (\ref{Generator}). In our study, we take $E=(-\infty, 0]$. Because the fish population extinction occurs in this interval.

For simplicity, we choose $D=(r_1,r_2)$ instead of the interval $D=(0,K)$, where $K$ is the carrying capacity of the fish population in the environment. We can rewrite equation \ref{EPLM2} as:
\begin{align}\label{EPLM3}
 rx(1-x)p_E'(x)+|h(x)|^{\alpha}\int_{\mathbb{R}\setminus{\{0\}}}(p_E(x+z)-p_E(x)-zp_E'(x)I_{|z|<1}(z))\nu(dz)=0, \tag{6.10}
  \end{align}
for $x\in(r_1,r_2)$; $p_{E}(x)=1$ for $x\in(-\infty,r_1]$ and $p_{E}(x)=0$ for $x\in[r_2,\infty)$.

The numerical algorithms of equation (\ref{EPLM2}) was done based on the scheme in the paper \cite{gao2014mean}, and its numerical simulation is similar to the mean exit time, so by the discretization method given in the paper by T. Gao \cite{gao2014mean}, we obtain the numerical solution of the escape probability.

The numerical solution of escape probability are ploted in figure \ref{EP2} with the noise intensity $\sigma$ and the stability index $\alpha$ varied. For fixed values of the growth rate $r$ and the noise intensity $\sigma$, and for the fish population $x \in (0,0.15)$,  the probability of fish extinction is small with $\alpha$ increases. While fish population $x \in (0.25, 1)$, the phenomenon is opposite. In other words probability of fish extinction is high with the same vales of $\alpha$, $\sigma$ and $r$, (see Fig. \ref{EP2}(a)). From this solution, we conclude that the interval  (0.15, 0.25) is a transition period. In the case $\alpha \in (1,2)$, the EP increases with the increase of stability index $\alpha$,(see Fig. \ref{EP2}(b)).In Fig. \ref{EP2}(c) when $\sigma \in (0,1)$ increases, the EP decreases with fixed values of $\alpha$ and $r$. As a result, we conclude that a large stability index $\alpha \in (1,2)$ induces larger escape probability , this means that the probability of the fish population goes to extinct is high but a larger positive noise intensity $\sigma$ favours smaller escape probability or the fish population is sustained.

\begin{figure}[h!]
    \centering
     \begin{subfigure}[b]{0.45\textwidth}
         \centering
         \includegraphics[width=\linewidth]{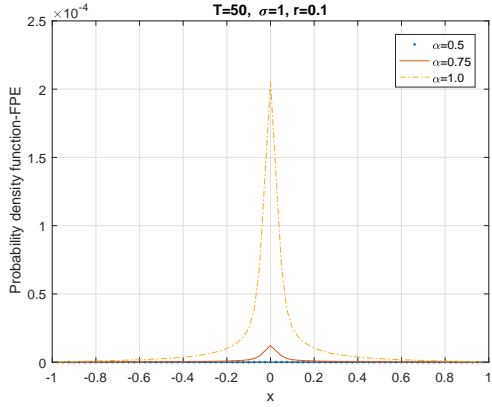}
    \caption{The PDF for $\alpha \in (0,1).$}
  \end{subfigure}
     \begin{subfigure}[b]{0.45\textwidth}
         \centering
        \includegraphics[width=\linewidth]{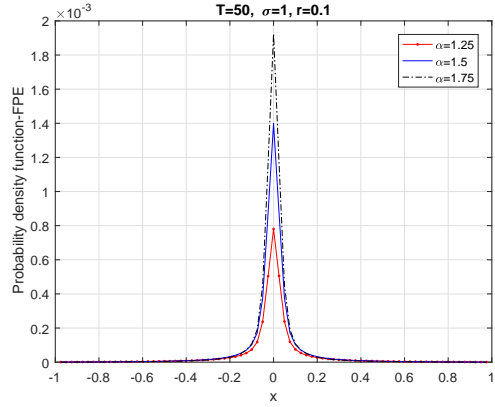}
    \caption{The PDF for  $\alpha \in (1,2)$. }
  \end{subfigure}
     \begin{subfigure}[b]{0.5\textwidth}
         \centering
     \includegraphics[width=\linewidth]{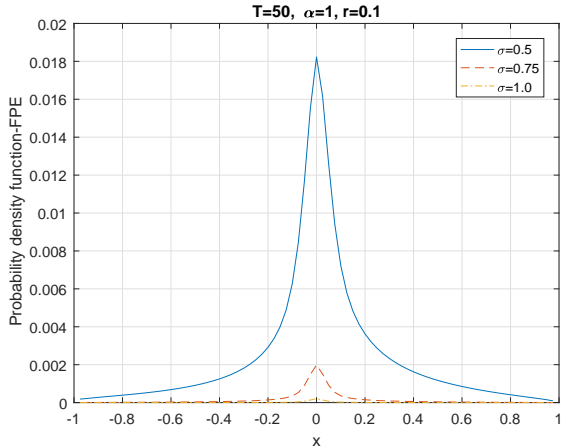}
    \caption{The PDF $\sigma \in (0,1)$. }
  \end{subfigure}
  \caption{  This figure shows the PDF of the FPE of SDE in (\ref{SDELM2}), with fixed $r=0.1$ and for different values of $\sigma$ and $\alpha$, (a) $\sigma = 1.0, \alpha \in (0,1)$ (b) $\sigma = 1.0, \alpha \in (1,2)$ , (c) $\sigma  \in (0,1), \alpha = 0.5$.}
  \label{FPE2}
\end{figure}

Since we take domain $D = (0,K)$ to be in the low concentration region, a smaller MET implies a higher likelihood for the fish population extinction( and vise versa ), and a larger EP indicates a higher likelihood for fish population extinction( and vise versa ).
\subsection{\textbf{Fokker-Plank equation}}

According the Fokker-Plank equation in Eq. (\ref{FPEG}), the Fokker-Planck equation for the stochastic differential equation in (\ref{SDELM2}), i.e., the probability density $p(x,t)$ for the solution process $X(t)$ with initial condition $X_0$ = $x_0$ and $p(x,0)=\sqrt {\frac{40}{\pi}}e^{-40x_0^2}$ satisfies the following nonlocal partial differential equation:
\begin{equation}\label{FPELM}
\frac{\partial p}{\partial t}=-\frac{\partial}{\partial x}(f(x)p(x,t))+\int_{\mathbb{R}\setminus{\{0\}}}[|h(x+z)|^{\alpha}p(x+z)-|h(x)|^{\alpha}p(x)]\nu (dz),\tag{6.11}
\end{equation}
$$p_t = -( r\,x(1-x)p)_x + |h(x)|^{\alpha}\int_{\mathbb{R}\setminus{\{0\}}}[|h(x+z)|^{\alpha}p(x+z)-|h(x)|^{\alpha}p(x)]\nu(dz).$$

To simulate the nonlocal Fokker-Planck equation (\ref{FPELM}), we apply a numerical finite difference method developed in Gao et al. \cite{gao2014mean}.

Figure \ref{FPE2}, shows the results for the probability density function of FPE under multiplicative symmetric $\alpha$-stable L\'evy motion. In Fig. \ref{FPE2}(a), when $\sigma = 1$ and $r= 0.1$ , the PDFs of FPE are larger corresponding to larger values of stability index $\alpha$, for example ( $\alpha = 0.5, 0.75, 1.0)$. When the non-Gaussianity index $\alpha$ lies between 1 and 2, i.e. ( $\alpha = 1.25, 1.5, 1.75)$, the PDF of FPE increases with in increase $\alpha$. (see Fig. \ref{FPE2}(b)). In Figure \ref{FPE2}(c), we can observe that the PDF of the FPE decreases with the increase in the noise intensity $\sigma$ while stability index is kept fixed at $\alpha=1$.

\section{Results}\label{7}

We analyse how the Gaussian noise intensity $\lambda$, the non-Gaussian noise intensity $\sigma$, and the stability index $\alpha$ affect the MET, EP and the behavior of the probability density function of the FPE of equation (\ref{SDEBMLM1}). Then we have explained the biological interpretation of the results based on our numerical experiments. This logistic differential equation model is monostable in some range of growth rate $r$ and carrying capacity $K$.
\emph{\subsection{\textbf{Results of stochastic logistic equation under Gaussian noise}}}
In this subsection, under Gaussian Brownian motion, we present MET, probability of exit, and stationary densities of the Fokker-Plank equation to observe the extinction and recovery time of stochastic logistic equation for different noise intensities as $x\rightarrow 0$.

 For Stochastic logistic system, we now examine the mean exit time, starting at $x \in (0,1),$ and reaching a new place or domain. In Figure (\ref{MeanExitTimeCaption}), using parameter values $\lambda=1$ and $\epsilon=0.001$, we compare a numerical approximation to the solution of Eqs. (\ref{METEq})-(\ref{METBCs}) (dashed red curve) to the solution ({YxminusY1}) of problem (\ref{METDomain01}) (solid blue curve) in which we used a truncated series approximation of $Y(x)$ in Eq. (\ref{YSeries}).
Starting from $x \in D=(0, \infty)$, we note that the expected exit time goes to infinity as $x \rightarrow 0$ when $0<\lambda<\sqrt{2}$.

When $\lambda > 0$, that is, when the noise is present, the equilibrium point can be reached in finite time since fluctuations guarantee that for some finite $t$, $X(t)$ will exceed the value 1. However, when $\lambda = 0$, in this case starting at $X(0) = x_0$ where $0 < x_0< 1$ can not reach the equilibrium point at 1 in finite time. When $\lambda < \sqrt 2$ the MET $u(x)$ from the interval $(0,1)$ at the right boundary is finite, but $\lim_{x\rightarrow 0}u(x)=\infty$. When $\lambda > \sqrt 2$, even though the left boundary is accessible, the MET from $(0,1)$ at this boundary is infinite. ( see Figure \ref{FPE1})

If $\lambda < \sqrt 2$, starting from $x > 0$, the probability of exit at x=0 in finite time is zero, so this boundary is not accessible. However, if $\lambda > \sqrt 2$, the probability of exit in finite time is greater than 0.

When $\sqrt r$  is large relative to $\tilde{\lambda}$ the probability density peaks near $x=1$ but as dispersion intensity $\tilde{\lambda}$ increases relative to $\sqrt r$ the probability mass density accumulates near $x = 0$. If $\lambda < \sqrt 2$, $\lim_{t\rightarrow\infty}X(t) = Z,$ where $Z$ is a random variable with probability density function given by Eq. (\ref{fpe2}). If $\lambda > \sqrt 2$, $q(x)=\lim_{t\rightarrow\infty}p(x,t/x_0)=\delta(x)$, the delta function with all probability mass concentrated at $x=0$.
\emph{\subsection{\textbf{Results on the stochastic logistic equation under non-Gaussian noise}}}
In this section, we have explained the effect of the parameters $\sigma$ and $\alpha$ on the three deterministic quantities, namely MET, EP and FPE under non-Gaussian noise. In the interval of stability index $\alpha \in (0,1)$ and density of the fish population $x \in (0,0.3)$,  we observed that the MET increases with fixed value of the growth rate $r=1.0$ and the noise intensity $\sigma=0.5$. While the fish population size $x \in(0.35,1)$, the numerical result is the same, i.e, the MET increases with the increase $\alpha$.(see Fig.\ref{MET2}(a)). When  $r=1.0$ and $\sigma=1.0$, the MET increases with increases in the stability index in ${\alpha}$ with the fish population size $x \in (0, 0.55)$. While $x \in (0.6,1)$, MET decreases with the increase in $\alpha$,.

Figure \ref{EP2} shows the numerical solution of escape probability with $\sigma$ and $\alpha$ are varied for non-Gaussian noise case. For fixed values of the growth rate $r$ and the noise intensity $\sigma$, we observe that, the probability of fish extinction is small with $\alpha \in (0,1)$) and for fish size in the interval $x \in (0,0.15)$ (see Fig. \ref{EP2}(a)). The initial density of fish size is larger than 0.25, the probability of fish extinction is high with $\alpha \in (0,1)$), ( Fig. \ref{EP2}(a)). In the case $\alpha \in (1,2)$, the EP increases with the increase of $\alpha$. When the values of the stability index $\alpha$ and the growth rate $r$ are kept fixed, the probability of fish extinction is small the noise intensity ($\sigma \in (0,1))$ increases.
An implication of this phenomenon in the stochastic logistic fish growth model can be understood as follows. If $\alpha \in (0,1)$, there is a smaller probability of extinction of the fish population with the interval $x\in (0,0.15)$. Contrary to this phenomenon, the probability of extinction of the fish population is high if the fish density $x$ is greater than $0.25$ and $\alpha \in(0,1)$ and  $\alpha \in (1,2)$. Figure \ref{EP2}(c) exhibits that EP decreases with the increasing $\sigma$ with fixed value of $\alpha$ and the growth rate $r$. This leads to the conclusion that larger noise intensity $\alpha$ indicates the extinction of the fish population is less likely. A large positive noise intensity $\sigma$ can induce small escape probability.

Figure \ref{FPE2} shows the PDF of FPE under multiplicative symmetric $\alpha$-stable L\'evy motion is dependent on the stability index, growth rate, and noise intensity. Figure. In \ref{FPE2}(a), for $\sigma = 1$ and $r$ = 0.1, we observe the PDF of FPE increases as $\alpha$, ($\alpha \in (0,1))$ increases. When the non-Gaussianity index is large, i.e.($\alpha \in (1,2)$), the PDF of FPE increases with the increase in $\alpha$. So a large $\alpha$ can induce larger PDF of FPE, (see Fig. \ref{FPE2}(b)). From Fig. \ref{FPE2}(c), we observe that the PDF of the FPE decreases for different values of the noise intensity $\sigma$ with fixed value of stability index $\alpha$.
\section{Conclusions} \label{8}
In summary, we have investigated stochastic logistic model of the fish population driven by both white noise and non-Gaussian noise. We have indeed proved the existence and uniqueness of a positive solution of our model under the Assumption 1- Assumption 4, and Theorem \ref{Thrm1} and Theorem \ref{Ther}. The dynamical properties of a fish population growth system are investigated based on the mean exit time, escape probability and Fokker-Plank equation.

The MET, EP, and FPE of the logistic model for a fish population with a symmetric $\alpha$-state L\'evy motion satisfy a nonlocal partial differential equation while in the Gaussian case, they satisfy local partial differential equation. We discuss the effects of the noise parameters on the three deterministic quantities MET, EP and FPE in detail. The multiplicative noise makes the problem difficult.

The MET for a stochastic logistic system quantifies how long, in expected sense, the fish population (or the system ) stays in a region in the state space.

We analyze the biological interpretation of the results based on the numerical experiments. From the biological perspective, these results tell us the following about the fish population growth.

A smaller MET indicates a higher likelihood for the fish population extinction ( and vice versa) and a larger escape probability implies a higher likelihood probability of extinction. In other words a higher escape probability from 0 ( left boundary, unstable state ) implies a higher probability of fish population extinction. Thus, a higher escape probability and a smaller mean exit time are not preferred in the fish population growth.

When $\lambda$ is sufficiently small, the dynamics of the system is primarily controlled by the drift. In this case the fish population exhibits slow exponential decay towards equilibrium which tends to increase the MET. When $\lambda$ is near 5.5, relatively larger fluctuations decrease the MET to a local minimum. As $\lambda$ increases, trajectories tend to stagnate and spend more time in the region near $x=0$.

When $\lambda \geq \sqrt2$, the denominator in Eq. (\ref{fpe2}) is infinite implying that $q(x)=0$ for $x>0$ in the limit as $t\rightarrow\infty$, and that all the probability mass has accumulated at $x=0$. In other words, $X(t)$ can exit the state space $(0,\infty)$ at $x=0$ when $\lambda \geq \sqrt2$ but the expected time to exit is infinite.

 This leads us to the conclusion that sample paths of the stochastic differential equation in Eq. (\ref{SDEMB2}) can not reach $x=0$ or $x=\infty$ in finite time as long as $\lambda \geq \sqrt2$. In other words, these boundaries are inaccessible. We further conjecture that the $x=0$ is accessible when $\lambda \geq \sqrt2$.

 If $\lambda $ is strictly larger than $\sqrt 2$, the probability of exit at the left endpoint from the domain $(0,L)$, $P(x,0,L) > 0$. This means that some fraction of the trajectories will never exit at the right endpoint $x = L,$ that is, $Pr\{\omega|X(t,\omega) < L $ for all $t > 0\}> 0.$

In order to get a low likelihood of extinction, we can tune the symmetric index $\alpha$ smaller to have a smaller EP. To get a higher MET, we can tune the noise intensity $\sigma$ smaller, and stability index $\alpha$ larger in the case of non-Gaussian noise.

 The results of escape probability lead to the conclusion that a larger stability index $\alpha$ is less likely leading to the extinction of the fish population. A large positive noise intensity of $\sigma$ can induce small escape probability. Furthermore, when the noise intensity of $\sigma$ is small, the EP is very sensitive to the initial density.
The probability for the fish population $x(t)$ to escape to the left of the domain is larger when the value of the stability index $\alpha $( $\alpha\in (1,2)$) increases. This suggests the fish population goes to extinction (see Fig. \ref{EP2}(b)). Thus we should not choose a large stability index. Because higher EP causes a higher probability of fish population extinction or die out..
\section{Acknowledgments}
The Authors would like to thank Yancai Liu and Xiaofan Li for their helpful discussions on numerical schemes. This work was partly supported by the NSFC grants- China 11801192, 11771449 and 11531006.

\end{document}